\documentclass[12pt]{amsart}

\usepackage{graphicx}
\usepackage[alphabetic]{amsrefs}
\usepackage[all]{xy}
\theoremstyle{definition}

\newtheorem{thm}{Theorem}
\newtheorem{defn}[thm]{Definition}

\newtheorem{lem}[thm]{Lemma}
\newtheorem{cor}[thm]{Corollary}
\newtheorem{exa}[thm]{Example}

\title{Series Parallel Linkages}

\author{James Cruickshank}
\email{james.cruickshank@nuigalway.ie}
\address{School of Mathematics,
Statistics and Applied Mathematics, National University of Ireland,
Galway, Ireland}

\author{Jonathan McLaughlin}
\email{DL92@o2.ie}
\address{School of Mathematics, Statistics and Applied Mathematics,
National University of Ireland, Galway, Ireland}


\date{\today}

\begin{document}

\begin{abstract}
We study spaces of realisations of linkages (weighted graphs) whose
underlying graph is a series parallel graph. In particular, we
describe an algorithm for determining whether or not such spaces are
connected.
\end{abstract}
\maketitle

\section{Introduction}

Let $G$ be a graph with vertex set $V$ and edge set $E$. Let $l:E
\rightarrow \mathbb R^{\geq 0}$ (where $\mathbb R^{\geq 0}$ denotes
the nonnegative real numbers). We will call $l$ a length function.
We will call such a pair, $(G,l)$, a linkage. Note that this is not
standard terminology. However, it seems appropriate in the given
context to formalise the intuition that a weighted graph is the
mathematical model for a mechanical linkage consisting of hinges and
bars that are constrained to move in a plane (we ignore the issue of
self intersections).

Given such a linkage
$L=(G,l)$ we define the space of planar configurations of $L$ as
follows: \[ C(L) = C(G,l) := \{p:V \rightarrow \mathbb R^2: |p(u)
-p(v)| = l(\{u,v\}) \forall \{u,v\} \in E\}\] where $|p(u)-p(v)|$
denotes the standard Euclidean distance between $p(u)$ and $p(v)$.
By definition, $C(L)$ is a subset of $\mathbb R^{2|V|}$ and thus
inherits a natural metric space structure. Observe that there is a
canonical action of the group  of orientation preserving isometries
of the plane on $C(L)$. We define the the moduli space of the
linkage, denoted by $M(L)$ or $M(G,l)$, to be the orbit space of
this action. It is easy to see that if $G$ is connected then $M(L)$
is a compact real algebraic variety. In general, it is difficult to
decide whether or not $M(L)$ is even nonempty. An element of $M(L)$
is called a realisation of the linkage $L$.

The problem of finding a realisation of $L$ is known as the molecule
problem - see \cite{MR1358807}. In the case where $M(L)$ is
nonempty, it is difficult to say much about the topology of this
space without imposing some restrictions on the structure of the
underlying graph $G$. The case where $G$ is a polygonal graph (i.e.
connected with every vertex of degree two) is quite well understood
and much is known about the topology of $M(L)$ in this case. For
example, we have the following (see Theorem 1.6 of \cite{MR2076003},
for example).

\begin{thm}
\label{polygonal} If $G$ is a polygonal graph, then
$M(L)$ is nonempty if and only if the longest edge has length at
most half the total length of all the edges. Moreover $M(L)$ is
connected if and only if the sum of the lengths of the second and third longest
edges is at most half of the total length of all the edges.
\end{thm}

Indeed, much more detailed information about the topology of $M(L)$
is available when $G$ is polygonal. The homotopical and homological
properties of these spaces are well understood - see
\cite{MR1133898}, \cite{MR1614965} or \cite{MR2076003}, for example.
For an overview of some of the theory of polygonal linkages, we
refer the reader to \cite{MR2455573}.

Our purpose in this paper is to study $M(L)$ where $G$ is a series
parallel graph (see below for definitions). We will show that it is
possible to easily determine, for a given series parallel graph $G$
and weight function $l$, whether or not $M(L)$ is nonempty and, in
the case when it is nonempty, whether or not it is connected. The
problem of deciding whether the space is connected or not is
connected with the motion planning problem in robotics. The motion
planning problem is concerned with the existence of a path betweeb
two configurations of a robot. If we think of our linkages as a
model for mechanical linkages, then the motion planning problem for
this particular type of ``robot" is equivalent to finding a
continuous path in $M(L)$ with specified endpoints.

We will show that for the class of series parallel graphs these
problems can be answered by considering a finite system of linear
inequalities in the edge lengths. In contrast, we note that even for
the complete graph on four edges, the smallest 2-connected graph
which is not series parallel, it is necessary to solve a polynomial
equation of total degree 6 (quartic in each variable) in the edge
lengths to determine whether or not a realisation exists.

Throughout this paper we adopt the convention that  $L=(G,l)$ and that $L_i=(G_i,l_i)$.

\section{Series Parallel Graphs}

In this section, we review the basic constructions and facts
concerning the class of series parallel graphs. First, we fix some
conventions regarding some standard graph theory. A graph is a pair
$(V,E)$ where $E$ is a multiset of unordered pairs of distinct
elements of $V$. Thus, in particular multiple edges with the same
endpoints are allowed. However loops are not allowed. A path graph
is a graph isomorphic to a graph with vertex set $\{1,\dots,n\}$
and edge set $\{ \{i,i+1\}:i=1,\dots,n-1\}$. A path linkage is a
linkage $(P,l)$ where $P$ is a path graph. A polygonal graph is a
graph isomorphic to a graph with vertex set $\{1,\dots,n\}$ and edge
set $\{ \{i,i+1\}:i=1,\dots,n-1\}\cup\{\{n,1\}\}$. A polygonal
linkage is a linkage $(G,l)$ where $G$ is a polygonal graph.

A two terminal graph (TTG) is an ordered triple $(G,s,t)$ where $s$
and $t$ are distinct vertices of $G$ called the source and the sink,
respectively. Collectively $s$ and $t$ are called the terminal
vertices of the TTG. Given  TTGs $(G_1,s_1,t_1)$ and $(G_2,s_2,t_2)$
we can define the series composition $(G_1,s_1,t_1)\circ
(G_2,s_2,t_2)$ to be the TTG \[(G_1\cup_{t_2\sim s_1}G_2,s_2,t_1)\]
where $G_1\cup_{t_2\sim s_1}G_2$ denotes the graph obtained by
identifying the vertices $t_2$ and $s_1$. Also we define the
parallel composition $(G_1,s_1,t_1)\| (G_2,s_2,t_2)$ to be the TTG
\[ (G_1 \cup_{s_1\sim s_2, t_1\sim t_2} G_2, s_1,t_1).\] See Figure
\ref{SPcomposition} for an illustration of these constructions.
Observe that the operation of parallel composition is a commutative
associative operation on the class of TTGs. Thus, in particular,
given TTGs $(G_i,s_i,t_i)$ for $i=1,\dots,n$ we can unambiguously
refer to the parallel composition
\[(G_1,s_1,t_1)\| \dots \|(G_n,s_n,t_n).\]
\begin{figure}\includegraphics{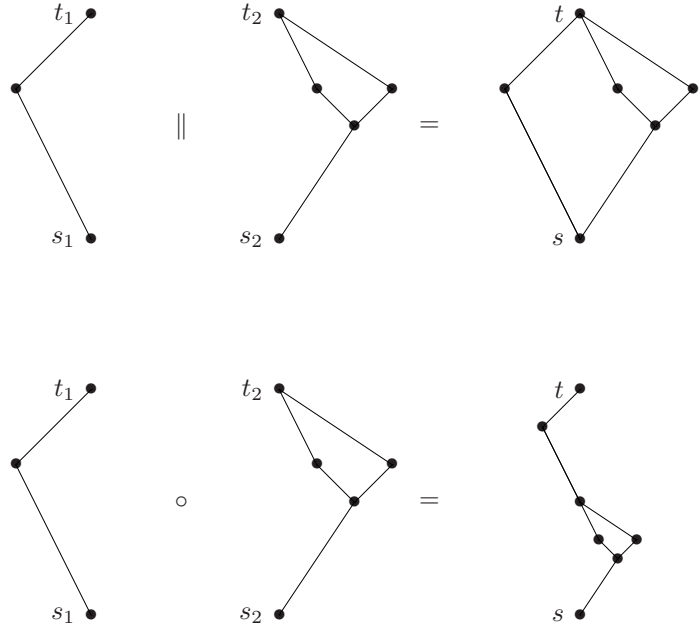}
\caption{\label{SPcomposition} Parallel and series composition of
two terminal graphs.}\end{figure} Let $K_2$ denote the complete
graph with vertex set $\{s,t\}$. We define the class of two terminal
series parallel graphs (TTSPGs) to be the smallest class of TTGs
that contains $(K_2,s,t)$ and that is closed under the operations of
series and parallel composition. A series parallel graph is a graph
$G$ such that $(G,s,t)$ is a TTSPG for some choice of vertices $s$
and $t$. Thus, for example, path graphs are series parallel. Also
polygonal graphs are series parallel, since a polygon is the
parallel composition of two paths. A series parallel linkage is a
linkage $(G,l)$ such that $G$ is a series parallel graph. We note
that the operations of parallel composition and series composition
extend in an obvious way to linkages - so it makes sense to refer to
the parallel composition $(L_1,s_1,t_1)\|(L_2,s_2,t_2)$ or the
series composition $(L_1,s_1,t_1)\circ(L_2,s_2,t_2)$, where $L_1$
and $L_2$ are linkages rather than graphs.

Observe that for a given series parallel graph, there may be many
possible choices of terminal vertices. However, the choice is not
completely arbitrary - some pairs of vertices cannot be the terminal
vertices of a given series parallel graph. For example, the
existence of a subgraph of $G$ homeomorphic to
\[\includegraphics{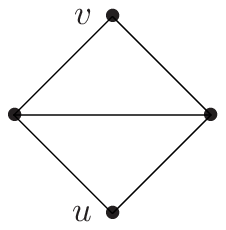}\]  implies that $(G,u,v)$
is not a TTSPG. There are other possible obstructions. For a more
detailed discussion of the possible choices of terminal vertices,
see \cite{MR1161075}.

The following lemma will prove useful for our analysis of the
connectedness of the moduli space of a series parallel linkage.
Recall that a graph $G$ is 2-connected if the complement of any
vertex is connected. Observe that a series parallel graph is
2-connected if and only if it cannot be expressed as a series
composition of proper subgraphs.

\begin{lem} \label{twoconnSPGlemma} Let $G$ be a 2-connected
series parallel graph. There are vertices $s$ and $t$ in $G$ such
that $(G,s,t)$ is a TTSPG and such that \[(G,s,t) = (P_1,s,t)\|
(P_2,s,t)\|(H,s,t),\] where $P_1$ and $P_2$ are paths joining $s$
and $t$ and $H$ is a (possibly empty) subgraph of $G$ such that
$(H,s,t)$ is a TTSPG. \end{lem}

\begin{proof}
Let $P$ be a subgraph of $G$ such that $P$ is a path and such that
every interior vertex of $P$ has degree 2 in $G$ (i.e. no other
edges of $G$ are incident to the interior of $P$). Let $s$ and $t$
be the endpoints of $P$. By an easy modification of the proof of
Lemma 9 in \cite{MR1161075}, we see that $(G,s,t)$ is a TTSPG (note
that the hypothesis of 2-connectedness is necessary at this point)
and thus $(G,s,t) = (P,s,t)\|(K,s,t)$. Here $K$ is the subgraph of
$G$ spanned by all the edges that are not in $P$. Now it also easy
to show (for example, by induction on the number of edges) that in
any series parallel graph that is not itself a path graph, it is
possible to find two distinct path subgraphs $P_1$ and $P_2$ with
common endpoints and such that no other edges of $G$ are incident
with any of the internal vertices of $P_1$ and $P_2$. Applying our
previous observation to $P_1$ and $P_2$ completes the proof of the
lemma.
\end{proof}

Series parallel graphs are a well studied class of graphs (see
\cite{MR849395} and \cite{MR0175809} for example). Of particular
interest to us is the following result of Belk and Connelly (see
\cite{MR2295049}). We say that a graph is $d$-realisable (where $d$
is a positive integer) if given any positive integer $n$ and any
function $f:V \rightarrow \mathbb R^n$, there exists a function $g:V
\rightarrow \mathbb R^d$ such that $|g(u)-g(v)| = |f(u)-f(v)|$ for
all edges $\{u,v\} \in E$. Intuitively, this means that any
embedding of $G$ into some (possibly high dimensional) Euclidean
space can be squashed into $\mathbb R^d$ so that the edge lengths
are preserved.

\begin{thm}\label{connellybeck}
(Belk, Connelly) A graph is 2-realisable if and only if it is a
series parallel graph. \end{thm}

Of course, knowing that a given graph $G$ is 2-realisable does not
tell us whether or not $(G,l)$ is realisable for a particular length
function $l$, nor does it tell us anything about the topology of
$M(G,l)$. However Theorem \ref{connellybeck} does suggest that the
class of series parallel graphs is an interesting class for which to
study the space $M(G,l)$.

\section{realisability}\label{sec:realisability}

Now suppose that $(G,s,t)$ is a TTSPG graph and that $l$ is a length
function on $G$. Let $L$ be the linkage $(G,l)$. Let  \[ [L,s,t] =
\{ |p(s)-p(t)|: p \in M(L)\}.\] Here
we are abusing notation somewhat by writing $p$ for an element of
$M(P)$, but also using $p$ to denote a particular representative in
$C(P)$ of the orbit under the action of orientation preserving
isometries of $\mathbb R^2$. However, this clearly does not cause
any problems with this definition as the quantity $|p(s)-p(t)|$ is
preserved by this action. We will consistently abuse notation in
this way throughout the remainder of the paper. In other words $[L,s,t]$ is the set
of all possible values of the distance between $p(s)$ and $p(t)$ as
$p$ varies over all realisations in $M(L)$. In the case where $L$ is
a path linkage, there is only one possible choice for the set of
terminal vertices, so we will write $[L]$ for $[L,s,t]$ in this
case.

Note, that $[L,s,t]$ could be empty. Indeed, the linkage is
realisable (i.e. $M(L)$ is nonempty) if and only if $[L,s,t]$ is
nonempty.

We will show that it is possible to easily compute $[L,s,t]$ for a
given TTSPG. Observe that in general it is difficult to compute the
set of possible distances between a pair of points as we vary over
all realisations of a (possibly non series parallel) graph. However for the special situation that
we consider, it is possible.

\begin{lem} Let $L_1 = (G_1,l_1)$ and let $L_2=(G_2,l_2)$ and let
$(G,s,t)=(G_1,s_1,t_1)\|(G_2,s_2,t_2)$. Then \[
[L,s,t] = [L_1,s_1,t_1] \cap [L_2,s_2,t_2]. \]
In particular, $L$ is realisable if and only if $[L_1,s_1,t_1] \cap [L_2,s_2,t_2]$
is nonempty.
\end{lem}

\begin{proof} If $x \in [L,s,t]$ then there is some $p \in M(L)$
such that $|p(s)-p(t)| = x$. For $i=1,2$, let $p_i = p|_{G_{i}}$.
Now $x=|p_i(s) - p_i(t)|$, so $ x\in [L_i,s,t]$. This shows that
$[L,s,t] \subseteq [L_1,s,t]\cap [L_2,s,t]$. For the other
inclusion, suppose that $x \in [L_1,s,t]\cap [L_2,s,t]$. So, for
$i=1,2$, there are realisations $p_i$ of $L_i$  such that $|p_i(s) -
p_i(t| = x$. Clearly, $p_1$ and $p_2$ together induce a realisation
$p$ of $L$ such that $|p(s)-p(t)|$. \end{proof}

For series compositions, we make the following definition.
\begin{defn}
Given intervals $[a,b]$ and $[c,d]$ with $0\leq a \leq b$ and $0\leq
c \leq d$, define the composition $[a,b] \circ [c,d]$ to be the
interval \[ [\max\{0,c-b,a-d\} ,b+d]\] If we write $\phi$ to denote
the empty interval, then we define $[a,b]\circ \phi := \phi$.
\end{defn}

Observe, for example, that $[a,b] \circ [c,d] = [0,b+d]$ if and only
if $[a,b]\cap[c,d]$ is nonempty.

\begin{lem} Let $L_1=(G_1,l_1)$ and $L_2= (G_2,l_2)$ be linkages such
that $[L_1,s_1,t_1]$ and $[L_2,s_2,t_2]$ are both closed intervals.
Let $(G,s_2,t_1)=(G_1,s_1,t_1)\circ(G_2,s_2,t_2)$.  Then \[
[L,s_2,t_1] = [L_1,s_1,t_1]\circ [L_2,s_2,t_2] \]
\end{lem}

\begin{proof} This follows immediately from the observation that
$x \in [L,s_2,t_1]$ if and only if there is $y\in [L_1,s_1,t_1]$ and
$z \in [L_2,s_2,t_2]$ such that $x$, $y$ and $z$ are the lengths of
the sides of a triangle.
\end{proof}

\begin{cor} Let $(G,s,t)$ be a TTSPG and let $L=(G,l)$. Then
$[L,s,t]$ is either empty or is a closed bounded interval of
$\mathbb R$.
\end{cor}

\begin{proof} This follows from a simple induction on the number of edges in $G$. \end{proof}

We note that there are efficient algorithms available for
recognizing series parallel graphs and for finding a series parallel
decomposition of a given series parallel graph (see \cite{MR1161075}
and \cite{MR652904}). Now, it is clear how to compute $[L,s,t]$ when
$G$ is a TTSPG. In particular, this allows us to easily determine
whether or not a given series parallel linkage is realisable.

\begin{exa} \label{LinkageExample}
Let $L$ be the series parallel linkage whose combinatorial structure
is indicated in the diagram below \[
\includegraphics{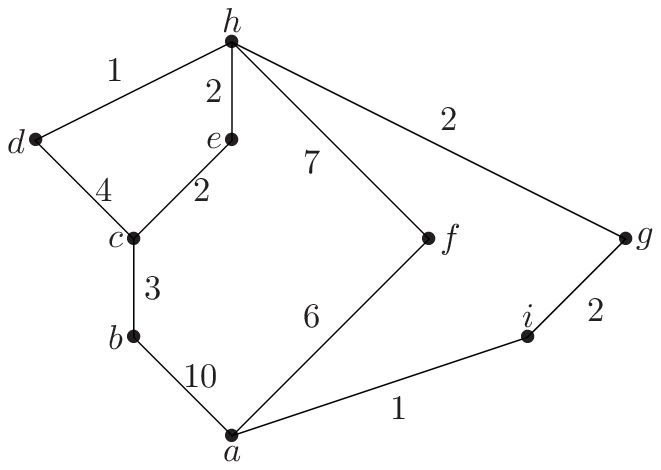}. \] The label on an edge
is the length of that edge. Consider the following five sublinkages
of $L$ \[ \begin{array}{ccc}
    P_1 = \includegraphics{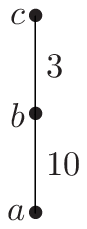} &
    P_2 = \includegraphics{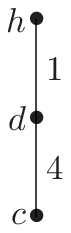} &
    P_3 = \includegraphics{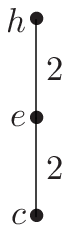} \\
    P_4 = \includegraphics{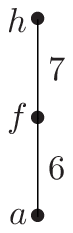}&     P_5 = \includegraphics{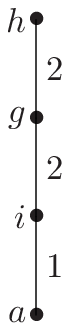} &
\end{array}\]
Clearly, \[(L,a,h) = ( ((P_2,c,h)\|(P_3,c,h))\circ (P_1,a,c))\|(P_4,a,h)\|(P_5,a,h).\]
Now $[P_1] = [7,13]$, $[P_2]=[3,5]$, $[P_3]=[0,4]$, $[P_4] = [1,13]$
and $[P_5] = [0,5]$. Therefore, \[ \begin{array}{rcl}
[L,a,h] &=& ( ([3,5]\cap [0,4])\circ[7,13])\cap[1,13] \cap[0,5] \\
&=& ([3,4]\circ[7,13])\cap[1,13]\cap[0,5] \\
&=& [3,17]\cap[1,13]\cap[0,5] \\
&=& [3,5] \end{array}\] In particular, the linkage $L$ is
realisable.
  \end{exa}

We conclude this section by showing that the realisability problem
for a given series parallel linkage can be answered by looking only
at the polygonal sublinkages of the given linkage. This is not
necessarily true for linkages whose underlying graph is not series
parallel. For example, consider the complete graph on four vertices where
each edge is given length 1. Every polygonal sublinkage of this
linkage is realisable in the plane but the complete linkage is not.
However, for series parallel graphs, we have the following.

\begin{cor}
Let $L = (G,l)$ be a series parallel linkage. Then $L$ is realisable
if and only if, for every polygonal subgraph $H$ of $G$, the linkage
$(H,l)$ is realisable. \end{cor}

\begin{proof}
It is obvious that if $L$ is realisable then every sublinkage of $L$
is also realisable. For the other implication, we argue by
contradiction. Suppose that $L$ is a counterexample to the statement
with the minimal number of edges. So $L$ is not realisable but every
polygonal sublinkage of $L$ is realisable. Note that the minimality
of $L$ ensures that every proper sublinkage of $L$ is realisable. In
particular $L$ cannot be decomposed as a series composition of
proper sublinkages. So there is some pair of vertices $s,t$ in $G$
such that $(G,s,t) = (G_1,s,t)\| (G_2,s,t)$, and such that $
[L_1,s,t]$ and $[L_2,s,t]$ are nonempty but $[L_1,s,t] \cap
[L_2,s,t]$ is empty.  Assume without loss of generality that
$[L_1,s,t]=[a_1,b_1]$ lies to the left of $[L_2,s,t]=[a_2,b_2]$
(i.e. $b_1<a_2$). Now we observe that there is some path graph $P_1$
joining $s$ to $t$ contained in $G_1$ such that $[P_1] =
[\alpha,b_1]$ and there is some path graph $P_2$ joining $s$ to $t$
contained in $G_2$ such that $[P_2] = [a_2,\beta]$. It is clear that
the polygonal linkage $(P_1,s,t)\| (P_2,s,t)$ is not realisable
which contradicts our assumption that all polygonal sublinkages of
$L$ are realisable.
\end{proof}

\section{Connectedness}

To understand the connectedness of $M(L)$ for a series parallel
linkage $L$,  we need to more precisely understand the relationship
between configurations of a path linkage and the corresponding distances between
the images of the terminal vertices.

Throughout this section let $P$ be a path linkage with $k$ edges and
suppose that $k \geq 2$. We suppose that all the edges of the linkage
$P$ have nonzero length.  Let $s$ and $t$ be the terminal vertices of
$P$. Let $\theta: M(P) \rightarrow [P]$, $\theta(p) := |p(s)-p(t)|$.
In this section we will show that $\theta$ has a certain lifting
property. The basic idea is to use Morse theory to analyse the
fibrewise structure of $\theta$. We remark that the differential
properties of the map $\theta$ are well understood (see
\cite{MR2125272}). It is differentiable at all points not in
$\theta^{-1}(0)$. Also the points where the derivative of $\theta$
vanishes are precisely the straight line configurations of $P$ (i.e
those points $p$ for which the set $\{p(v):v \in V_P\}$ lies in an
affine line in $\mathbb R^2$). We will say that $p \in M(P)$ is a
critical point of $\theta$ if either $\theta(p) = 0$ or $\theta'(p)
= 0$.


The basic question that we now consider is this. Suppose that $p$
and $q$ are two configurations of a path linkage $P$. Clearly, since
$M(P)$ is pathwise connected (it is homeomorphic to $(S^1)^{k-1}$),
it is possible to find a path (in the sense of topological spaces)
in $M(P)$ that connects $p$ to $q$. However, suppose that the motion
of the endpoints of $P$ is specified. Is it possible to find a path
in $M(P)$ connecting $p$ and $q$ so that the endpoints of $P$ move
in a specified way? Theorems \ref{liftingtheorem1} and \ref{liftingtheorem2} below will be the key
to constructing paths in $M(L)$ when $L$ is a series parallel
linkage.

First, we need some notation to describe a particular subset of $[P]$.

\begin{defn} We define $\nabla(P)$ to be the following subset
of $[P]$. \[ \nabla(P) = \{ x \in [P]: \theta^{-1}(x) \text{ is connected}\}.\]
\end{defn}

Note that if $P$ has just two edges of length $l_1$ and $l_2$, then
$\nabla(P) = \{|l_1-l_2|, l_1+l_2\}$ (i.e. $\nabla(P)$ consists of
two points). When $P$ has more than two edges, $\nabla(P)$ is union
of at most two closed intervals, as the following analysis shows.

Let $l_1,\dots,l_k$ be the lengths of the edges in $P$, with $k\geq3$.
We suppose for the moment that $l_1 \geq
l_2\geq \dots \geq l_k$. (Note that permuting the edge lengths does
not affect the homeomorphism type of $\theta^{-1}(x)$.)

In the following lemma we are using the convention that for $b<a$,
$[a,b]$ is the empty set.

\begin{lem} \label{connectivityrange} Let $S = \sum_{i=1}^k l_i$.
Then   $\nabla(P)$ is \[[P]\cap \biggl([2(l_2+l_3)-S,l_3]\cup
[l_3,\min\{l_1,S-2l_2\}]\cup[\max\{l_1,S-2(l_1+l_2)\},S]\biggr).
\] \end{lem}

\begin{proof} This is a straightforward application of Theorem \ref{polygonal} to
the polygonal linkage obtained by adjoining the edge $\{s,t\}$ to
$P$ and extending the length function $l$ by defining $l(\{s,t\}) =
x$, where $x$ is and arbitrary element of $[P]$. \end{proof}

In particular, Lemma \ref{connectivityrange} shows that for a given
path linkage, $P$, it is very straightforward to calculate
$\nabla(P)$.

\begin{exa} Suppose $P$ has 3 edges and that $l_1=l_2=l_3 = 1$.
Then $\nabla(P) = [1,3]$ which is a proper subset of $[P] = [0,3]$.
\end{exa}


\subsection{Lifting Properties of $\theta$}\label{sec:lifting}

\begin{thm} \label{liftingtheorem1} Let $p,q \in M(P)$ and suppose
that neither $\theta(p)$ nor $\theta(q)$ are critical values of
$\theta$. Let $\alpha:[0,1] \rightarrow [P]$ be continuous and
suppose that $\alpha(0) = |p(s)-p(t)| = \theta(p)$ and $\alpha(1) =
|q(s)-q(t)| = \theta (q)$. If $\text{im}(\alpha) \cap \nabla(P)$ is
non empty then there exists a continuous lift $\tilde{\alpha}:[0,1]
\rightarrow M(P)$ such that $\tilde{\alpha}(0) = p$,
$\tilde{\alpha}(1) = q$ and $\theta \circ \tilde{\alpha} = \alpha$.
\end{thm}

Note that Theorem \ref{liftingtheorem1} requires that neither $p$
nor $q$ lie in the preimage of a critical value. In order to remove
this hypothesis, we must tighten the requirements on $\alpha$. In
particular, we may require $\alpha$ remains stationary for some
positive amount of time near 0 or near 1. More precisely, we have
the following.

\begin{thm} \label{liftingtheorem2}
Let $p,q \in M(P)$. Let $\alpha:[0,1] \rightarrow [P]$ be continuous
and suppose that $\alpha(x) = \theta(p)$ for $x \in [0,\epsilon]$
for some $\epsilon >0$,  and that $\alpha(x) = \theta (q)$ for $x
\in [1-\delta,1]$ for some $\delta >0$. If $\text{im}(\alpha) \cap
\nabla(P)$ is non empty then there exists a continuous lift
$\tilde{\alpha}:[0,1] \rightarrow M(P)$ such that
$\tilde{\alpha}(0) = p$, $\tilde{\alpha}(1) = q$ and $\theta \circ
\tilde{\alpha} = \alpha$. \end{thm}

In order to prove Theorems \ref{liftingtheorem1} and
\ref{liftingtheorem2} we will need to understand the fibrewise structure of
the map $\theta: M(P) \rightarrow [P]$. This map has been studied
 by previous authors using the techniques of
Morse theory. In particular, Shinamoto and Vanderwaart have given a
very clear account of this theory in \cite{MR2125272} (their
notation is somewhat different to ours).
 We can summarise the situation as follows. Let $W = M(P) -
\theta^{-1}(0)$. Then $\theta|_{W}: W \rightarrow [P]$ is a
differentiable function. Moreover, in this restricted domain, $\theta$ has
finitely many critical points, all of which are nondegenerate. If we
also include $0$ as a critical value of $\theta :M(P) \rightarrow
[P]$, then there are finitely many critical values $0 \leq a_0 <a_1
< \dots <a_s = S$. For $i=0,\dots s-1$, let $M_i =
\theta^{-1}([a_i,a_{i+1}])$. By standard results of Morse theory
(see \cite{MR0163331}), we know that for each $i = 0,\dots, s-1$
there is a smooth closed $k-2$ dimensional manifold $\Sigma_{i}$
such that
\[M_{i} \equiv \frac{[a_i,a_{i+1}] \times \Sigma_i}{\sim}\] where
$\sim$ collapses some subsets of $\{a_i\}\times \Sigma_i$ to points
and also some subsets of $\{a_{i+1}\}\times \Sigma_{i}$ to points.
In other words $M_i$ is obtained by taking $[a_i,a_{i+1}]\times
\Sigma_i$ and making some identifications over the endpoints $a_i$
and $a_{i+1}$. Indeed, we can be more explicit over the non zero
critical points. In those cases the identifications are obtained by
collapsing some finite number of embedded spheres in $\Sigma_i$. In
the case where $a_0 = 0$, the collapsing over $a_0$ can be a little
more complicated, but that does not affect the validity of our
arguments below.

So, for each $i = 1,\dots,s-1$, we have a commutative diagram \[ \xymatrix{
    M_i \ar[rr]^{\equiv}\ar[dr]_{\theta}& & \frac{[a_i,a_{i+1}]
    \times \Sigma_i}{\sim}\ar[dl]^{\text{projection}}\\
    & [a_i,a_i+1] & } \]
Clearly all of the non zero critical values of $\theta$ are
contained in $\nabla(P)$ (this is an easy exercise for the reader!).
Moreover, it is known (see \cite{MR2076003}) that if $x \notin \nabla(P)$,
then $\theta^{-1}(x)$ is the disjoint union of two copies of
$(S^1)^{k-2}$. In other words, $\Sigma_i$ is disconnected if and
only if $\Sigma_i = (S^1)^{k-2} \sqcup (S^1)^{k-2}$.

See Figure \ref{fibrewisefig} for an illustration of the structure
of $\theta$. For the purposes of illustration we have represented
pieces of $M(P)$ as two dimensional surfaces, even though $M(P)$ is
actually a torus of dimension $k-1$. However, Figure
\ref{fibrewisefig} does give a reasonably faithful picture of how
the fibres of $\theta$ behave. In the example illustrated in the
figure the open interval $(a_j,a_{j+1})$ lies in the complement of
$\nabla(P)$. The curves drawn in the interior of $M_0$, $M_i$ and
$M_j$ are meant to represent the fibres of $\theta$ over the points
$x$, $y$ and $z$ respectively.
\begin{figure} \includegraphics{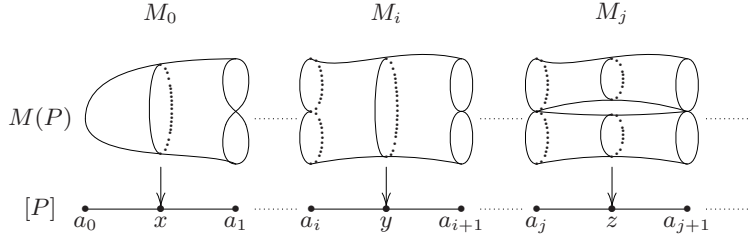}
\caption{ \label{fibrewisefig} The fibrewise structure of the map
$\theta$.}\end{figure}

Now, it is clear how we should prove Theorems \ref{liftingtheorem1}
and \ref{liftingtheorem2}.\medskip

\noindent {\em Proof of Theorem \ref{liftingtheorem1}.} We
make the following observations. The projection \[
\frac{[a_i,a_{i+1}]\times \Sigma_i}{\sim} \rightarrow
[a_i,a_{i+1}]\] has sections. Indeed given points $z_1$ and $z_2$
that lie in the same path component of $\Sigma_i$ and given distinct
points $x_1$ and $x_2$ in $[a_i,a_{i+1}]$ there exists a continuous section
\[\sigma:[a_i,a_{i+1}]\rightarrow \frac{[a_i,a_{i+1}]\times
\Sigma_i}{\sim}\] such that $\sigma (x_1) = (x_1,z_1) $ and $\sigma
(x_2) = (x_2,z_2)$.

By combining $\alpha$ with a judicious use of these sections, it is
clear that we can find the required lift of $\alpha$. For the sake
of completeness, we have included the details of this argument
below. However these details are rather tedious and not particularly
enlightening. Thus, if the reader if sufficiently convinced by the
arguments already presented, he may, at this point, skip the rest of
this proof.

We must consider several different cases. First, let us deal with
the case where $p$ and $q $ happen to lie in same fibre of $\theta$.
If $\alpha(x) = \theta(p) = \theta(q) $ for all $x \in [0,1]$ then
the conclusion is clearly true, as by assumption we must have
$\theta (p) \in \nabla(P)$, and therefore we can lift $\alpha$ by
choosing a path within $\theta^{-1}(\theta(p))$ that connects $p$
and $q$. If $\alpha$ is not a constant function then we can choose
some $b \in [0,1]$ such that $b\neq \theta (p)$ but so that
$\alpha([0,b])$ is contained within one of the open intervals
$(a_i,a_{i+1})$. Now it is clear that we can lift
$\alpha|_{[0,b]}:[0,b]\rightarrow [P]$ since $\theta$ restricts to a
trivial fibre bundle over $\alpha([0,b])$. Thus we are left the
problem of lifting $\alpha|_{[b,1]}:[b,1]\rightarrow [P]$ with
specified lifts of $b$ and of 1. In other words, we have reduced to
case where $p$ and $q$ lie in different fibres of $\theta$.

Suppose now that $p$ and $ q$ lie in different fibres of $\theta$.
Also suppose that $\theta(p) \in \nabla(P)$ (similar arguments apply
if $\theta (q) \in \nabla(P)$). Since we have assumed that
$\theta(p)$ is not a critical value, $\theta(p)$ must in fact lie in
the interior of $\nabla(P)$. Therefore, by concatenating local
sections over $[a_i,a_{i+1}]$ of the type described above, we can
find a (global) section $\gamma: [P] \rightarrow M(P)$ such that
$\gamma (\theta (p) ) = p $ and $\gamma(\theta(q)) = q $. Let
$\tilde \alpha = \gamma \circ \alpha$. Clearly $\tilde \alpha$ is
the required lift in this case.

Finally, we consider the case where $\theta(p) \notin\nabla(P)$ and
$\theta(q) \notin \nabla (P)$. Choose some $c \in [0,1]$ such that
$\alpha(c) \in \nabla(P)$ (our hypotheses guarantee the existence of
at least one such $c$). Now we choose a point $r$ in the fibre
$\theta^{-1}(\alpha(c))$. If $\alpha(c)$ is in the interior of
$\nabla(P)$, we can choose $r$ arbitrarily within the fibre.
However, if $\alpha(c)$ happens to be on the boundary of $\nabla(P)$
(and is therefore also a critical value of $\theta$), we must be
more selective in our choice of $r$. In this case we choose $r$ to a
critical point of $ \theta $. Now, once we have chosen $r$ in this
way, we can find two global sections $\gamma_{1}$ and $\gamma_2$ of
$\theta$, such that $\gamma_1(\alpha(c))= \gamma_2(\alpha(c)) = r$,
$\gamma_1(\theta(p)) = p$ and $\gamma_2(\theta (q)) = q$. Now, for $
x\in [0,c]$, let $\tilde\alpha (x) = \gamma_1(\alpha(x))$ and for $x
\in [c,1]$, let $\tilde\alpha(x) = \gamma_2(\alpha(x))$. One readily
checks that $\tilde\alpha$ is the required lift of $\alpha$ in this
case.
 \qed
\medskip

\noindent {\em Proof of Theorem \ref{liftingtheorem2}.} In the case
that one of $p$ or $q$ lies in a critical fibre (i.e the preimage of
one of the $a_i$s), then it may be necessary to first move to a
different point in that fibre to ensure that when we lift along
sections, we end up in the right path component of subsequent
fibres. The hypotheses of Theorem \ref{liftingtheorem2} allow the
intervals $[0,\epsilon]$ and $[1-\delta,1]$ to carry out this
adjustment within the fibre. \qed \medskip

We will also need the following lifting result later. Its proof is
again an straightforward consequence of the fibrewise structure of
$\theta $ described above, so we shall leave the reader to fill in
the details in this case.

\begin{thm} \label{liftingtheorem3}
Let $p \in M(P)$ and suppose that $\alpha :[0,1]\rightarrow [P]$ is
a continuous function such that $\alpha (0) = \theta (p)$. There is
some continuous lift $\tilde{\alpha} : [0,1 ] \rightarrow M(P)$ such
that $\tilde{\alpha}(0) = p$ and $\theta \circ \tilde{\alpha} =
\alpha$. \end{thm}

Now we show that any path in $M(P)$ that connects two different path
components of a fibre $\theta^{-1}(x)$ must pass through $
\theta^{-1}(\nabla(P))$.

\begin{lem}\label{connlemma} Suppose that $x \notin \nabla(P)$ and let $p$ and $q$ be
two realisations of $P$ that lie in different components of
$\theta^{-1}(x)$. Suppose that $\alpha
:[0,1]\rightarrow M(P)$ is a continuous function such that
$\alpha(0) = p$ and $\alpha(1) = q$. Then there is
some $c \in [0,1]$ such that $\theta(\alpha(c)) \in
\nabla(P)$.
\end{lem}

\begin{proof}
It is clear from the above description of $\theta$ that there must
exist $c$ such that $\alpha(c) = a_j$ for some critical value $a_j$
of $\theta$. However, as remarked above, $a_j \in \nabla(P)$.
\end{proof}

We conclude this section by observing that if $x\notin \nabla(P)$,
if $p \in \theta^{-1}(x)$ and if $\tau$ is any orientation reversing
isometry of  $\mathbb R^2$, then $\tau\circ p $ and $p$ lie in
different path components of $\theta^{-1}(x)$.

\subsection{Determining the connectedness of the moduli space}

Now we present a method for checking the connectedness of $M(L)$
when $L = (G,l)$ is a series parallel linkage. First observe that we
may as well restrict our attention to the case where the graph $G$
is a 2-connected series parallel graph. If $G$ is not 2-connected,
then it can be decomposed into a series composition of 2-connected
series parallel graphs. It is clear that $M(L)$ is connected if and
only the moduli space of each of the series components of $L$ is
connected.

Let $L$ be a 2-connected series parallel linkage such that $M(L)$ is
not empty. Recall that, by Lemma \ref{twoconnSPGlemma}, we can find
vertices $u$ and $v$ such that $(G,u,v)$ is a TTSPG and such that \[
(G,u,v) = (P_1,u,v)\|(P_2,u,v)\| \dots \|(P_n,u,v) \| (K,u,v)\]
where $(K,u,v)$ is a sub TTSPG of $(G,u,v)$ and
where each $P_i$ is a path joining $u$ and $v$, and $n \geq 2$.  See
Figure \ref{parallelpaths} for an illustration of this situation.
\begin{figure}\includegraphics{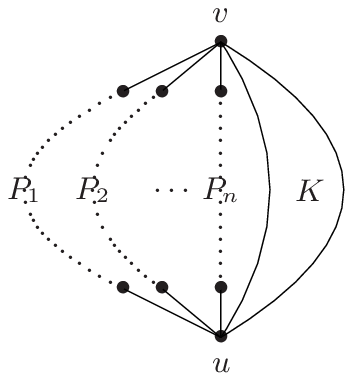}\caption{\label{parallelpaths}}\end{figure}
We will write $L_K$ to denote the sublinkage $(K,l|_{K})$ of $L$.

\begin{thm} \label{SPconone}
With notation as above, if $\nabla(P_i) \cap [L,u,v]$ is empty for
some $i$ then $M(L)$ is disconnected. \end{thm}

\begin{proof}
Let $p \in M(L) $. We can construct another realisation of $L$ by
reflecting the vertices $P_i$ in a line through $p(u)$ and $p(v)$.
Call this realisation $q$. Now if $\alpha:[0,1]\rightarrow M(L)$ is
any path, then by assumption $|\alpha(x)(u)-\alpha(x)(v)| \notin
\nabla(P_i)$ for all $x \in [0,1]$. Therefore, by Lemma
\ref{connlemma} and the observation that immediately follows that
lemma, $\alpha$ cannot be a path that connects $p$ and $q$.
\end{proof}

What can we say about the connectedness of $M(L)$ if the hypothesis
of Theorem \ref{SPconone} is not satisfied, in other words if
$\nabla(P_i) \cap [L,u,v]$ is non empty for all $i$?

In this case, we construct a path linkage $Q$ as follows. Suppose
that $[P_i] = [a_i,b_i]$ for $i = 1,\dots,n$. Let $a = \max\{a_i\}$.
Let $b= \min\{b_i\}$ (note that $b \geq a$ since $(G,l)$ is
realisable). Now let $Q$ be a path linkage with four edges
$e_1,\dots,e_4$ and assign length $l_i$ to edge $e_i$ as follows;
$l_1 = \frac{a+b}2$ and $l_2=l_3=l_4 = \frac{b-a}6$.

\begin{lem} \label{Qproperty} $[Q]=[a,b]$ and $\nabla(Q) = [a,b]$. \end{lem}
\begin{proof} The first statement is obvious and
the second statement follows immediately from Lemma
\ref{connectivityrange} \end{proof}

Let $s$ and $t$ be the terminal vertices of $Q$ and define a linkage
$L_1$ by \[( L_1, u,v) = (Q,s,t) \| (K_L,u,v).\] In other words $
L_1$ is obtained from $L$ by replacing all the $P_i$s by the single
path linkage $Q$. Note that $L_1$ has a strictly smaller series
parallel decomposition in terms of path linkages than $L$ does.

\begin{thm} \label{SPcontwo}
With the notation as above, suppose that $\nabla (P_i) \cap [L,u,v]$
is nonempty for each $i=1,\dots,n$. Then $M(L)$ is connected if and
only if $M(L_1)$ is connected.\end{thm}

\begin{proof} We first observe that, by construction, $[Q] = [P_1]\cap \dots
\cap[P_n]$. It follows that \[ \text{im}(M(L)\rightarrow M(L_K)) =
\text{im}(M(L_1)\rightarrow M(L_K))\] where $M(L)\rightarrow M(L_K)$
and $ M(L_1)\rightarrow M(L_K)$ are the canonical maps induced by
restriction.

Now, suppose that $M(L_1)$ is connected and let $p$ and $q$ be
elements of $M(L)$. We must show that there is a path in $M(L)$
joining $p$ and $q$. We construct this path in several stages.
First, by our observations above, we can choose some realisations
$p_1$ and $q_1$ of $L_1$ that agree with $p$ and $q$ on $L_K$. Now
since $M(L_1)$ is connected, there is some path $\alpha:[0,1]
\rightarrow M(L_1)$ such that $\alpha(0) = p_1$ and $\alpha(1) =
q_1$. Now we can apply Theorem \ref{liftingtheorem3} to construct a
path $\tilde\alpha : [0,1] \rightarrow M(L) $ such that
$\tilde\alpha(0) = p$ and $\tilde\alpha$ agrees with $ \alpha$ on
vertices of $K$. We just define $\tilde \alpha (t) (v)  =
\alpha(t)(v)$ for all vertices $v\in K$. To lift $\tilde\alpha$ to
$M(L)$, we use Theorem \ref{liftingtheorem3} (once for each $P_i$).
In particular $\tilde\alpha(1)|_K = q|_K$.

Of course, it may happen that for some or all of the $P_i$s,
$\tilde\alpha(1)|_{P_i}\neq q|_{P_i}$. So we have to concatenate
other paths onto the end of $\tilde\alpha$ to ``correct" it on the
$P_i$s. We can do this one $P_i$ at as time as follows. Let $x =
|q(u)-q(v)|$. By assumption $\nabla (P_i) \cap [L,u,v]$ is nonempty,
so there is some path $\beta:[0,1]\rightarrow [L,u,v]$ such that
$\beta(0) = x = \beta(1)$ and such that $\beta(y) \in \nabla(P_i)$
for some $y \in[0,1]$. Moreover, we can certainly choose $\beta$ so
that it is stationary in a neighbourhood of $0$ and in a
neighbourhood of $1$. Now, it is clear that by applying Theorem
\ref{liftingtheorem2} to $\beta$ we can find some
$\overline\beta:[0,1] \rightarrow M(L)$ such that $\overline\beta(0)
= \tilde\alpha(1)$, $\overline\beta(1)|_{P_i} = q|_{P_i}$ and
$\overline\beta(1)$ agrees with $\tilde\alpha(1)$ for vertices that
are not in $P_i$. Concatenating $\tilde\alpha $ and $\overline\beta$
``corrects" the final position of vertices of $P_i$. We can repeat
this process for all the $P_i$s, if necessary, and we eventually end
up with the required path in $M(L)$ connecting $p$ and $q$.

The converse can be proved in much the same way. Suppose that $M(L)$
is connected and let $p_1$ and $q_1$ be points in $M(L_1)$. We can
find $p$ and $q$ in $M(L)$ that agree with $p_1$ and $q_1$ on $L_K$.
Since $M(L)$ is connected, we can find a path $\alpha :[0,1]
\rightarrow M(L_K)$ such that $\alpha$ connects $p_1|_{L_K}$ and
$q_1|_{L_K}$ and such that $\text{im}(\alpha)$ is contained in the
image of the natural map $M(L) \rightarrow M(L_K)$.  Now, using
Theorem \ref{liftingtheorem3}, we can lift $\alpha$ to a path
$\tilde\alpha:[0,1]\rightarrow M(L_1)$ and using Theorem
\ref{liftingtheorem2} we can correct $\tilde{\alpha}(1)$ so that it
agrees with $q_1$ on $Q$ as necessary. Note that Lemma
\ref{Qproperty} ensures that the hypotheses of Theorem
\ref{liftingtheorem2} are satisfied in this situation.

\end{proof}

Theorems \ref{SPconone} and \ref{SPcontwo} form the basis of a
simple recursive algorithm for deciding whether or not $M(L)$ is connected for
a 2-connected series parallel linkage. We informally describe this
algorithm by the following sequence of steps. We assume that $M(L)$
is nonempty.
\begin{enumerate}
    \item Find a parallel decomposition of the form
    \[ [L,u,v] = (P_1,u,v)\|(P_2,u,v)\|\dots \|(P_n,u,v) \| (K,u,v)\] where $n \geq 2$.
    \item Compute $[L,u,v]$ using the methods described in Section
    \ref{sec:realisability}. Compute $\nabla(P_i)$ for each $i$ using Lemma
    \ref{connectivityrange}.
    \item If $\nabla(P_i) \cap [L,u,v]$ is empty
    for any $i$, then $M(L)$ is not connected and we can stop.
    \item If $\nabla(P_i) \cap [L,u,v]$ is non empty for all $i$, and $K$ is empty then $M(L)$
    is connected and we can stop.
    \item If $\nabla(P_i) \cap [L,u,v]$ is non empty for all $i$, and $K$ is non empty then
    construct the linkage $L_1$ as
    described above and go back to Step (1) with linkage $L_1$ as the input.
\end{enumerate}

\begin{exa}
Let $L$ be the same linkage that we considered in Example
\ref{LinkageExample} and let $P_1$, $P_2$, $P_3$, $P_4$ and $P_5$ be
the sublinkages described earlier. Now observe that \[ (G,c,h) =
(P_2,c,h)\|(P_3,c,h)\|( ((P_4,a,h)\|(P_5,a,h))\circ (P_1,c,a)). \]
It is easy to check (as in Example \ref{LinkageExample}) that \[
[L,c,h] = [3,4]  \] Moreover, $\nabla(P_2) = \{3,5\}$ and
$\nabla(P_3) = \{0,4\}$. Thus, in this case, the hypotheses of
Theorem \ref{SPcontwo} are satisfied. The linkage $L_1$ looks like
\[
\includegraphics{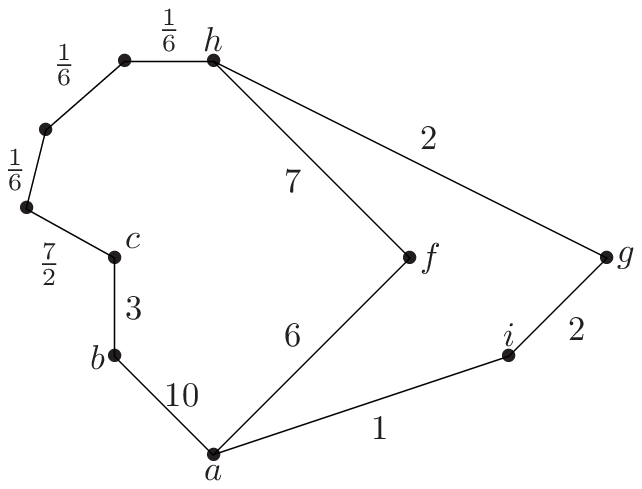}\] Now, one computes that
$[L_1,a,h] = [3,5]$. However, $\nabla(P_4) = \{1,13\}$ which does not
meet $[3,5]$. Therefore, by Theorem \ref{SPconone}, $M( L_1)$ is
disconnected. Therefore by Theorem \ref{SPcontwo}, $M(L)$ is
disconnected.

Observe that the linkage $L_1$ in this example has the property that
every polygonal sublinkage has connected moduli space, but that
$M(L_1)$ is disconnected.
\end{exa}

\subsection{Remarks} We observe that the path lifting results described
in Section \ref{sec:lifting} do not in general hold for linkages
that are not series parallel. In \cite{McLaughlinThesis}, examples
are given to demonstrate this. This is one of the reasons why series
parallel linkages are easier to understand.

We also remark that it is sometimes possible to adapt our methods to
understand linkages that are not series parallel. It may be that a
linkage can be series parallel decomposed into smaller linkages,
which while not themselves series parallel, are amenable to analysis
by other methods. In this case our results may still have some
value. Again, see \cite{McLaughlinThesis} for examples.

\section{Acknowledgements}

We would like to thank Javier Aramayona for many helpful comments
and suggestions.



\end{document}